\documentclass[12pt]{article}
\usepackage[a4paper,margin=1in]{geometry}
\usepackage{amssymb}
\usepackage{amsmath}
\usepackage{amsfonts}
\usepackage{amsthm}
\usepackage{color}
\usepackage{float}
\usepackage[all]{xy}
\usepackage{scalerel}
\usepackage{mathabx}
\usepackage{delimset}
\usepackage{mathtools}
\usepackage{enumitem} 
\usepackage{microtype} 
\usepackage{hyperref} 
\usepackage{cleveref}

\newcommand{\R}{\mathbb{R}}
\newcommand{\Q}{\mathbb{Q}}

\newcommand{\U}{\mathrm{U}} 
\newcommand{\charpoly}{\mathrm{P}} 
\newcommand{\T}{\mathrm{T}} 
\newcommand{\B}{\mathcal{B}} 
\newcommand{\D}{\mathrm{D}} 
\newcommand{\Du}{\D_{\U}} 
 
\newcommand{\W}{\mathrm{W}}

\newtheorem{theorem}{Theorem}

\newtheorem{proposition}{Proposition}
\newtheorem{corollary}{Corollary}

\theoremstyle{definition} 
\newtheorem{definition}[theorem]{Definition} 
\theoremstyle{remark} 
\newtheorem{remark}[theorem]{Remark}

\begin{document}

\title{\textbf{The $\U$-Derivative Associated with Linear Recurrence Sequences}}
\author{Ronald Orozco L\'opez}

\newcommand{\Addresses}{{% additional braces for segregating \footnotesize
  \bigskip
  \footnotesize

  %R.~Orozco, \textsc{Departamento de Matematicas, Universidad de los Andes, 
  % Carrera 1 N. 18A-12 Bogot\'a, Colombia}\par\nopagebreak
  \textit{E-mail address}, R.~Orozco: \texttt{rj.orozco@uniandes.edu.co}
  
}}

\maketitle
%\tableofcontents

\begin{abstract} 
We associate a canonical operator to every linear recurrence sequence admitting a distinct-root Binet representation \(\U_n=\sum_{j=1}^{k}A_j\lambda_j^n. \) The construction is based on the spectral Binet operator \(\B_\U=\sum_{j=1}^{k}A_jT_{\lambda_j}, \) where \(T_{\lambda_j}f(x)=f(\lambda_jx)\) is the dilation operator associated with the characteristic root \(\lambda_j\). The corresponding recurrence derivative is defined by \(\D_\U=x^{-1}\B_\U.\) Its fundamental property is the monomial identity \(\D_\U(x^n)=\U_nx^{n-1},\) which establishes a direct correspondence between recurrence sequences and degree-lowering operators. We develop the operator theory of \(\D_\U\), including iteration, a spectral product rule, the associated Leibniz defect, and a quotient rule. The construction reveals a natural bifurcation determined by the value \(\U_0\), producing the Ward and non-Ward regimes as intrinsic operatorial phenomena. The product structure is shown to be governed by \(k-2\) residual spectral channels, leading to notions of spectral cancellation and spectral balance. As an application, we obtain a complete qualitative classification of second-order recurrence calculi. Using Horadam sequences as a universal second-order model, we prove that every such calculus is determined by the invariant triple \((\W_0,\Delta, Z_\W),\) where \(\W_0\) identifies the Ward/non-Ward type, \(\Delta\) determines the spectral geometry, and \(Z_\W\) determines regularity or singularity.
\end{abstract} 
% To be completed after the literature review. 

\section{Introduction}

Generalized calculi associated with numerical sequences have played an important role in finite operator theory, $q$-calculus, and related areas of discrete analysis. Classical examples include Jackson's $q$-derivative \cite{jackson1,jackson2,jackson3}, Ward's calculus of sequences \cite{ward}, and more recent calculi based on Fibonacci and Lucas sequences \cite{pashaev,orozco,orozco2}. A common feature of these theories is that ordinary integers are replaced by an alternative sequence governing the action of a derivative-like operator. In Ward's framework, for example, the basic sequence is assumed to satisfy \[ \U_0=0, \qquad \U_1=1, \qquad \U_n\neq0 \quad (n\ge1), \] and the resulting calculus is developed from these arithmetic data. Similar constructions appear throughout the literature under various specialized choices of sequences. The purpose of the present paper is to reverse this point of view. Instead of beginning with a particular factorial sequence, we start from an arbitrary linear recurrence sequence possessing a distinct-root Binet representation 
\[ \U_n=\sum_{j=1}^{k}A_j\lambda_j^n, \] 
and ask whether its spectral representation can be lifted canonically to an operator acting on functions. We show that this is indeed possible. To every recurrence sequence we associate the spectral Binet operator 
\[ \B_\U=\sum_{j=1}^{k}A_jT_{\lambda_j}, \] 
where \[ T_{\lambda_j}f(x)=f(\lambda_jx) \] is the dilation operator associated with the characteristic root \(\lambda_j\). The corresponding recurrence derivative is then defined by \[ \D_\U=x^{-1}\B_\U. \] The resulting operator satisfies the fundamental identity \[ \D_\U(x^n)=\U_nx^{n-1}, \] which recovers the recurrence sequence as the coefficient system governing monomial lowering. This operator-theoretic approach reveals structures that do not appear explicitly in the traditional formulation of recurrence calculi. First, the distinguished value \(\U_0\) produces an intrinsic bifurcation \[ \U_0=0 \qquad\text{versus}\qquad \U_0\neq0, \] leading respectively to the Ward and non-Ward regimes. Second, multiplication is governed by a spectral product rule whose correction term is controlled by the residual characteristic channels. This leads naturally to a theory of Leibniz defects, spectral cancellation, and spectral balance. The paper culminates in a complete qualitative classification of second-order recurrence calculi. Using Horadam sequences as a universal second-order model, we show that every such calculus is determined by the three invariants \[ (\W_0,\Delta, Z_\W), \] where \(\W_0\) determines the Ward/non-Ward type, \(\Delta\) determines the spectral geometry, and \(Z_\W\) determines regularity or singularity. The principal contributions of the paper are therefore: 
\begin{enumerate} 
\item the construction of the spectral Binet operator; 
\item the definition of the recurrence derivative \(\D_\U\); 
\item the fundamental monomial identity; 
\item the recovery of recurrence data from operator action; 
\item the spectral product rule and the associated Leibniz defect; \item a quotient rule adapted to the spectral setting; 
\item a complete classification of second-order recurrence calculi. \end{enumerate} 
Throughout the paper we restrict attention to recurrence sequences with distinct characteristic roots. Singular and confluent regimes are considered only insofar as they are necessary for the completeness of the second-order classification and will be developed systematically in subsequent work.

\section{Linear recurrences and Binet data}\label{sec:linear-recurrences-binet}

We begin by fixing the recurrence and spectral data required throughout
the paper. Let $\mathbb{K}$ be a field and let
$
\U=(\U_n)_{n\geq 0}
$
be a nonzero sequence satisfying the homogeneous linear recurrence
\begin{equation}
\label{eq:linear-recurrence}
\U_{n+k}
=
a_1\U_{n+k-1}
+a_2\U_{n+k-2}
+\cdots
+a_k\U_n,
\qquad n\geq 0,
\end{equation}
where
\begin{equation}
\label{eq:nonzero-last-coefficient}
a_k\neq 0.
\end{equation}
The condition $a_k\neq 0$ ensures that the recurrence in
\eqref{eq:linear-recurrence} has effective order $k$. The associated
characteristic polynomial is
\begin{equation}
\label{eq:characteristic-polynomial}
\charpoly(t)
=
t^k-a_1t^{k-1}-a_2t^{k-2}-\cdots-a_k.
\end{equation}

The recurrence coefficients $a_1,\ldots,a_k$ determine the recurrence
law, whereas the initial data
\begin{equation}
\label{eq:initial-data}
\U_0,\U_1,\ldots,\U_{k-1}
\end{equation}
select a particular solution. In particular, two sequences may satisfy
the same recurrence relation and hence have the same characteristic
polynomial while corresponding to different initial data.

Since
\begin{equation}
\label{eq:P-at-zero}
\charpoly(0)=-a_k\neq 0,
\end{equation}
zero is not a characteristic root. Throughout the main part of the paper
we impose the following simple-spectrum hypothesis.

\medskip
\noindent
\textbf{Standing spectral hypothesis.}
The polynomial $\charpoly$ splits over $\mathbb{K}$ into $k$ pairwise
distinct roots
\begin{equation}
\label{eq:distinct-characteristic-roots}
\lambda_1,\ldots,\lambda_k.
\end{equation}
Thus
\begin{equation}
\label{eq:factorized-characteristic-polynomial}
\charpoly(t)
=
\prod_{j=1}^{k}(t-\lambda_j),
\end{equation}
where
\begin{equation}
\label{eq:root-conditions}
\lambda_i\neq\lambda_j
\quad (i\neq j),
\qquad
\lambda_j\neq 0
\quad (1\leq j\leq k).
\end{equation}

For each characteristic root $\lambda_j$, the geometric sequence
$(\lambda_j^n)_{n\geq0}$ satisfies \eqref{eq:linear-recurrence}.
Indeed, the identity
$
\charpoly(\lambda_j)=0
$
gives
\begin{equation}
\label{eq:characteristic-root-identity}
\lambda_j^k
=
a_1\lambda_j^{k-1}
+a_2\lambda_j^{k-2}
+\cdots
+a_k,
\end{equation}
and multiplication by $\lambda_j^n$ yields
\begin{equation}
\label{eq:geometric-recurrence}
\lambda_j^{n+k}
=
a_1\lambda_j^{n+k-1}
+a_2\lambda_j^{n+k-2}
+\cdots
+a_k\lambda_j^n.
\end{equation}

The distinctness of the characteristic roots implies that these
geometric solutions are linearly independent. This gives the spectral
representation that underlies the operator construction developed in the
next section.

\begin{theorem}[Binet representation]
\label{thm:Binet-representation}
Assume that \eqref{eq:distinct-characteristic-roots}--\eqref{eq:root-conditions}
hold. Then there exist unique coefficients
$
A_1,\ldots,A_k\in\mathbb{K}
$
such that
\begin{equation}
\label{eq:Binet-representation}
\U_n
=
\sum_{j=1}^{k}A_j\lambda_j^n,
\qquad n\geq0.
\end{equation}
Conversely, every sequence of the form \eqref{eq:Binet-representation}
satisfies the recurrence \eqref{eq:linear-recurrence}.
\end{theorem}

\begin{proof}
For every $j$, the sequence $(\lambda_j^n)_{n\geq0}$ satisfies
\eqref{eq:linear-recurrence} by \eqref{eq:geometric-recurrence}.
Therefore every sequence of the form
\begin{equation}
\label{eq:general-geometric-combination}
V_n
=
\sum_{j=1}^{k}A_j\lambda_j^n
\end{equation}
also satisfies \eqref{eq:linear-recurrence}.

It remains to choose the coefficients so that $V_n$ has the prescribed
initial data. Requiring
$
V_r=\U_r,
\qquad
r=0,\ldots,k-1,
$
gives the linear system
\begin{equation}
\label{eq:Vandermonde-system}
\begin{pmatrix}
1 & 1 & \cdots & 1\\
\lambda_1 & \lambda_2 & \cdots & \lambda_k\\
\lambda_1^2 & \lambda_2^2 & \cdots & \lambda_k^2\\
\vdots & \vdots & & \vdots\\
\lambda_1^{k-1} & \lambda_2^{k-1} & \cdots
& \lambda_k^{k-1}
\end{pmatrix}
\begin{pmatrix}
A_1\\
A_2\\
\vdots\\
A_k
\end{pmatrix}
=
\begin{pmatrix}
\U_0\\
\U_1\\
\vdots\\
\U_{k-1}
\end{pmatrix}.
\end{equation}
The coefficient matrix is the Vandermonde matrix associated with
$\lambda_1,\ldots,\lambda_k$, whose determinant is
\begin{equation}
\label{eq:Vandermonde-determinant}
\prod_{1\leq i<j\leq k}(\lambda_j-\lambda_i).
\end{equation}
By the distinct-root hypothesis, the quantity in
\eqref{eq:Vandermonde-determinant} is nonzero. Hence
\eqref{eq:Vandermonde-system} has a unique solution
$A_1,\ldots,A_k$.

The resulting sequence $(V_n)_{n\geq0}$ and $\U$ satisfy the same
order-$k$ recurrence and agree at the first $k$ indices. Uniqueness of
solutions of \eqref{eq:linear-recurrence} with prescribed initial data
therefore gives
$
V_n=\U_n,
\qquad n\geq0.
$
This proves \eqref{eq:Binet-representation}. The converse follows from
linearity and the fact that every geometric component
$\lambda_j^n$ satisfies the recurrence.
\end{proof}

The preceding theorem separates two types of data. The characteristic
roots
$
\lambda_1,\ldots,\lambda_k
$
are determined by the recurrence coefficients, while the Binet
coefficients
$
A_1,\ldots,A_k
$
depend on the particular solution selected by the initial conditions.

To make this relation explicit, define
\begin{equation}
\label{eq:initial-vector}
\mathbf{u}
=
\begin{pmatrix}
\U_0\\
\U_1\\
\vdots\\
\U_{k-1}
\end{pmatrix},
\qquad
\mathbf{A}
=
\begin{pmatrix}
A_1\\
A_2\\
\vdots\\
A_k
\end{pmatrix},
\end{equation}
and let
\begin{equation}
\label{eq:Vandermonde-matrix}
V(\lambda_1,\ldots,\lambda_k)
=
\begin{pmatrix}
1 & 1 & \cdots & 1\\
\lambda_1 & \lambda_2 & \cdots & \lambda_k\\
\lambda_1^2 & \lambda_2^2 & \cdots & \lambda_k^2\\
\vdots & \vdots & & \vdots\\
\lambda_1^{k-1} & \lambda_2^{k-1} & \cdots
& \lambda_k^{k-1}
\end{pmatrix}.
\end{equation}
Then
\begin{equation}
\label{eq:initial-spectral-correspondence}
\mathbf{u}
=
V(\lambda_1,\ldots,\lambda_k)\mathbf{A},
\end{equation}
and, since the Vandermonde matrix is invertible,
\begin{equation}
\label{eq:spectral-initial-correspondence}
\mathbf{A}
=
V(\lambda_1,\ldots,\lambda_k)^{-1}\mathbf{u}.
\end{equation}
Thus the initial values and the Binet coefficients provide two coordinate
systems for the same recurrence solution.

Evaluating \eqref{eq:Binet-representation} at $n=0$ gives the identity
\begin{equation}
\label{eq:U0-spectral-sum}
\U_0
=
\sum_{j=1}^{k}A_j.
\end{equation}
This apparently elementary relation will play a distinguished role in the
operator theory. In particular, the cases $\U_0=0$ and $\U_0\neq0$
produce different behavior of the associated $\U$-derivative on constants.

The data introduced in this section are precisely those needed for the
operator construction: the characteristic roots determine the available spectral modes, and the initial conditions determine their coefficients. In the next section, the geometric modes $\lambda_j^n$ will be replaced by dilation operators, producing the spectral Binet operator and the associated $\U$-derivative. The characteristic roots and the Binet coefficients constitute the spectral data from which the associated operator will be constructed.

\section{Binet operator and U-derivative}\label{sec:Binet-operator-U-derivative}

We now transfer the Binet representation from recurrence sequences to
operators. Throughout this section, the hypotheses and notation of
Section~\ref{sec:linear-recurrences-binet} remain in force. In particular, $\U=(\U_n)_{n\geq0}$ is a nonzero recurrence sequence whose characteristic polynomial has $k$ pairwise distinct nonzero roots
$\lambda_1,\ldots,\lambda_k$, and
\begin{equation}
\label{eq:Binet-representation-operator-section}
\U_n =\sum_{j=1}^{k}A_j\lambda_j^n,
\qquad n\geq0.
\end{equation}

The essential observation is that each geometric spectral mode
$\lambda_j^n$ has a natural operator counterpart given by dilation.
This correspondence allows the spectral representation of $\U$ to
be transferred directly to the polynomial algebra. 

\subsection{Dilation operators and the Binet operator}

For each $\lambda\in\mathbb K$, define the dilation operator
\begin{equation}
\label{eq:dilation-operator}
\T_\lambda f(x)=
f(\lambda x).
\end{equation}
Whenever the underlying function space is stable under the substitution
$x\mapsto\lambda x$, the operator $\T_\lambda$ acts linearly on that
space. In particular,
\begin{equation}
\label{eq:dilation-polynomials}
\T_\lambda:\mathbb K[x]\longrightarrow\mathbb K[x].
\end{equation}

Its action on monomials is diagonal:
\begin{equation}
\label{eq:dilation-monomial}
\T_\lambda x^n
=
\lambda^n x^n,
\qquad n\geq0.
\end{equation}
Thus, for every characteristic root $\lambda_j$,
\begin{equation}
\label{eq:dilation-characteristic-mode}
\T_{\lambda_j}x^n
=
\lambda_j^n x^n.
\end{equation}

Equation~\eqref{eq:dilation-characteristic-mode} is the operator analogue
of the geometric mode appearing in the Binet representation. This suggests
the following construction.

\begin{definition}[Spectral Binet operator]
\label{def:spectral-Binet-operator}
The spectral Binet operator associated with the recurrence sequence $\U$
is defined by
\begin{equation}
\label{eq:spectral-Binet-operator}
\B_{\U}=\sum_{j=1}^{k}A_j\T_{\lambda_j}.
\end{equation}
Equivalently,
\begin{equation}
\label{eq:spectral-Binet-operator-explicit}
\B_{\U}f(x)=\sum_{j=1}^{k}A_jf(\lambda_jx).
\end{equation}
\end{definition}

The notation $\B_{\U}$ is important. Although the characteristic roots
are determined by the recurrence law, the coefficients $A_j$ depend on
the initial data. Consequently, the Binet operator is attached to the
particular recurrence sequence $\U$, not merely to its characteristic
polynomial.

\begin{proposition}
\label{prop:Binet-operator-linearity}
The operator $\B_{\U}$ is linear on every vector space on which all the
dilations $\T_{\lambda_1},\ldots,\T_{\lambda_k}$ are defined.
\end{proposition}

\begin{proof}
Let $f,g$ belong to such a vector space and let
$\alpha,\beta\in\mathbb K$. Then
\begin{align*}
\B_{\U}(\alpha f+\beta g)
&=
\sum_{j=1}^{k}
A_j\T_{\lambda_j}(\alpha f+\beta g) \\
&=
\sum_{j=1}^{k}
A_j
\left(
\alpha\T_{\lambda_j}f
+
\beta\T_{\lambda_j}g
\right) \\
&=
\alpha\B_{\U}f
+
\beta\B_{\U}g.
\end{align*}
\end{proof}

The Binet representation is recovered exactly from the action of
$\B_{\U}$ on the monomial basis.

\begin{proposition}[Diagonal action of the Binet operator]
\label{prop:Binet-monomial-action}
For every $n\geq0$,
\begin{equation}
\label{eq:Binet-monomial-action}
\B_{\U}x^n
=
\U_nx^n.
\end{equation}
\end{proposition}

\begin{proof}
Using \eqref{eq:dilation-characteristic-mode} and the Binet representation
\eqref{eq:Binet-representation-operator-section}, we obtain
\begin{align*}
\B_{\U}x^n
=
\sum_{j=1}^{k}
A_j\T_{\lambda_j}x^n 
=
\sum_{j=1}^{k}
A_j\lambda_j^n x^n 
=
\left(
\sum_{j=1}^{k}A_j\lambda_j^n
\right)x^n 
=
\U_nx^n.
\end{align*}
\end{proof}

Thus the monomial basis diagonalizes $\B_{\U}$, with eigenvalue sequence
$(\U_n)_{n\geq0}$. In particular, the recurrence sequence is encoded in
the spectrum of the Binet operator relative to the monomial basis.

\begin{corollary}
\label{cor:Binet-operator-recovery}
The action of $\B_{\U}$ on the monomial basis uniquely determines the
sequence $\U$.
\end{corollary}

\begin{proof}
For every $n\geq0$, equation
\eqref{eq:Binet-monomial-action} uniquely determines $\U_n$ as the
scalar satisfying
$
\B_{\U}x^n=\U_nx^n.
$
\end{proof}

The operator $\B_{\U}$ therefore transfers the discrete spectral data
of the recurrence sequence to a degree-preserving operator on
$\mathbb K[x]$.

\subsection{Definition of the U-derivative}

The Binet operator preserves monomial degree. To obtain a
derivative-type operator, we compose it with the degree-lowering
multiplication operator
\begin{equation}
\label{eq:M-x-minus-one}
M_{x^{-1}}f(x)
=
x^{-1}f(x).
\end{equation}
On polynomials,
\begin{equation}
\label{eq:M-x-minus-one-domain}
M_{x^{-1}}:
\mathbb K[x]
\longrightarrow
x^{-1}\mathbb K[x]
\subseteq
\mathbb K[x,x^{-1}].
\end{equation}

\begin{definition}[$\U$-derivative]
\label{def:U-derivative}
The $\U$-derivative associated with the recurrence sequence $\U$ is
the operator
\begin{equation}
\label{eq:U-derivative-operator}
\D_{\U}
=
M_{x^{-1}}\B_{\U}.
\end{equation}
Hence
\begin{equation}
\label{eq:U-derivative-explicit}
\D_{\U}f(x)
=
\frac{1}{x}
\sum_{j=1}^{k}A_jf(\lambda_jx).
\end{equation}
\end{definition}

At the purely algebraic level,
\begin{equation}
\label{eq:U-derivative-polynomial-Laurent}
\D_{\U}:
\mathbb K[x]
\longrightarrow
x^{-1}\mathbb K[x]
\subseteq
\mathbb K[x,x^{-1}].
\end{equation}
The question of when $\mathbb K[x]$ itself is invariant under
$\D_{\U}$ will follow from the monomial action below.

As with the Binet operator, $\D_{\U}$ depends on the particular recurrence
sequence. Two different solutions of the same recurrence have the same
characteristic roots but may have different Binet coefficients, and hence
may define different $\U$-derivatives.

\begin{proposition}[Linearity]
\label{prop:U-derivative-linearity}
The operator $\D_{\U}$ is linear on its domain.
\end{proposition}
\begin{proof}
Follows from Definition \ref{def:U-derivative} and Proposition \ref{prop:Binet-monomial-action}.
\end{proof}

The definition of $\D_{\U}$ is spectral: it is obtained from the Binet
representation by replacing each geometric mode with the corresponding
dilation operator and subsequently lowering degree by one. Its
derivative-like behavior is therefore a consequence of the recurrence
structure rather than an independently imposed monomial rule.

\subsection{The fundamental monomial identity}

We now obtain the basic identity on which the subsequent operator theory
is built.

\begin{theorem}[Fundamental monomial identity]
\label{thm:fundamental-monomial-identity}
For every $n\geq0$,
\begin{equation}
\label{eq:fundamental-monomial-identity}
\D_{\U}x^n
=
\U_nx^{n-1}.
\end{equation}
\end{theorem}

\begin{proof}
Using Definition~\ref{def:U-derivative}, we have
\begin{align*}
\D_{\U}x^n
=
\frac{1}{x}
\sum_{j=1}^{k}
A_j(\lambda_jx)^n 
=
\left(
\sum_{j=1}^{k}
A_j\lambda_j^n
\right)x^{n-1}.
\end{align*}
By the Binet representation,
$
\sum_{j=1}^{k}A_j\lambda_j^n
=
\U_n,
$
and hence
$
\D_{\U}x^n
=
\U_nx^{n-1}.
$
\end{proof}

For ordinary differentiation,
\begin{equation}
\label{eq:ordinary-derivative-comparison}
\frac{d}{dx}x^n
=
nx^{n-1},
\end{equation}
whereas the recurrence-induced operator satisfies
\eqref{eq:fundamental-monomial-identity}. Thus $\U_n$ occupies, at the
level of monomial lowering, the role played by $n$ in ordinary
differentiation.

This analogy concerns only the monomial action. It does not imply that
$\D_{\U}$ satisfies the classical Leibniz rule, quotient rule, or any
other identity of ordinary differential calculus. Those properties must be
derived from the spectral representation of the operator.

By linearity, the fundamental identity immediately determines the action
of $\D_{\U}$ on arbitrary polynomials.

\begin{proposition}[Action on polynomials]
\label{prop:U-derivative-polynomials}
Let $f(x)=\sum_{n=0}^{N}c_nx^n$. Then
\begin{equation}
\label{eq:U-derivative-polynomial}
\D_{\U}f(x)
=
\sum_{n=0}^{N}c_n\U_nx^{n-1}.
\end{equation}
Equivalently,
\begin{equation}
\label{eq:U-derivative-polynomial-separated}
\D_{\U}f(x)
=
c_0\U_0x^{-1}
+
\sum_{n=1}^{N}c_n\U_nx^{n-1}.
\end{equation}
\end{proposition}

\begin{proof}
By linearity and Theorem~\ref{thm:fundamental-monomial-identity},
\begin{align*}
\D_{\U}f(x)
=
\D_{\U}
\left(
\sum_{n=0}^{N}c_nx^n
\right) 
=
\sum_{n=0}^{N}
c_n\D_{\U}x^n 
=
\sum_{n=0}^{N}
c_n\U_nx^{n-1}.
\end{align*}
Separating the term $n=0$ gives
\eqref{eq:U-derivative-polynomial-separated}.
\end{proof}

The constant monomial is particularly important. Setting $n=0$ in
\eqref{eq:fundamental-monomial-identity} gives
\begin{equation}
\label{eq:U-derivative-constant}
\D_{\U}(1)
=
\U_0x^{-1}.
\end{equation}
Since
\begin{equation}
\label{eq:U0-Binet-coefficients}
\U_0
=
\sum_{j=1}^{k}A_j,
\end{equation}
the action on constants records the total spectral weight of the Binet
representation.

\begin{corollary}[Polynomial invariance criterion]
\label{cor:polynomial-invariance}
The polynomial algebra $\mathbb K[x]$ is invariant under
$\D_{\U}$ if and only if $\U_0=0$

\end{corollary}

\begin{proof}
Suppose first that $\U_0=0$. Then
\eqref{eq:U-derivative-polynomial-separated} becomes
\begin{equation}
\label{eq:U-derivative-polynomial-Ward-form}
\D_{\U}f(x)
=
\sum_{n=1}^{N}c_n\U_nx^{n-1},
\end{equation}
which belongs to $\mathbb K[x]$. Hence
$
\D_{\U}\bigl(\mathbb K[x]\bigr)
\subseteq
\mathbb K[x].
$

Conversely, assume
$
\D_{\U}\bigl(\mathbb K[x]\bigr)
\subseteq
\mathbb K[x].
$
Since $1\in\mathbb K[x]$, equation
\eqref{eq:U-derivative-constant} implies that
$
\U_0x^{-1}\in\mathbb K[x].
$
This is possible only if $\U_0=0$.
\end{proof}

Thus the single recurrence value $\U_0$ determines whether the natural
polynomial domain is preserved by the $\U$-derivative. The cases
$\U_0=0$ and $\U_0\neq0$ will be distinguished systematically later;
for the moment, the essential point is that this dichotomy follows directly
from the spectral construction rather than from an additional assumption.

Finally, the fundamental identity shows that the operator retains the entire
recurrence sequence.

\begin{proposition}[Recovery of the recurrence sequence]
\label{prop:recovery-recurrence-from-DU}
The action of $\D_{\U}$ on the monomial basis uniquely determines
$\U_0,\U_1,\U_2,\ldots$.
\end{proposition}

\begin{proof}
For every $n\geq0$,
\begin{equation}
\label{eq:recovery-coefficient}
\D_{\U}x^n
=
\U_nx^{n-1}.
\end{equation}
Therefore $\U_n$ is precisely the coefficient of $x^{n-1}$ in
$\D_{\U}x^n$, and the complete sequence is recovered from the monomial
action of the operator.
\end{proof}

\begin{corollary}
\label{cor:equality-U-derivatives}
Let $\U=(\U_n)_{n\geq0}$ and $\mathrm{V}=(\mathrm{V}_n)_{n\geq0}$
be recurrence sequences for which the corresponding operators
$\D_{\U}$ and $\D_{\mathrm{V}}$ are defined. If
\begin{equation}
\label{eq:equality-operators-monomials}
\D_{\U}x^n
=
\D_{\mathrm{V}}x^n
\qquad
\text{for every }n\geq0,
\end{equation}
then
\begin{equation}
\label{eq:equality-recurrence-sequences}
\U_n=\mathrm{V}_n
\qquad
\text{for every }n\geq0.
\end{equation}
\end{corollary}

The spectral Binet operator is therefore a degree-preserving encoding of the
recurrence sequence, while the $\U$-derivative is its degree-lowering
counterpart. 

\section{Ward and non-Ward behaviour}

The first structural distinction associated with a recurrence sequence is determined by the distinguished value \(\U_0\). Recall that 
\begin{equation} \label{eq:Du-on-constant-recall} 
\Du(1) = \U_0x^{-1}. 
\end{equation} 
Consequently, the action of the operator on the constant monomial is completely determined by the initial value of the recurrence sequence. 

\begin{definition}[Ward regime]\label{def:ward-regime} 
A recurrence sequence \(\U\) is said to belong to the \emph{Ward regime} if 
\begin{equation} \label{eq:ward-regime} 
\U_0=0. 
\end{equation} 
\end{definition} 

\begin{definition}[Non-Ward regime] \label{def:nonward-regime} 
A recurrence sequence \(\U\) is said to belong to the \emph{non-Ward regime} if \begin{equation} \label{eq:nonward-regime} \U_0\neq0. \end{equation} 
\end{definition} 

The distinction between these two regimes is intrinsically operatorial. It does not depend on the recurrence order or the characteristic roots. It is determined solely by the behavior of the operator at degree zero. 

\begin{proposition}[Operator characterization of the Ward regime] \label{prop:ward-characterization} 
The following statements are equivalent: 
\begin{enumerate} 
\item \(\U_0=0\); 
\item \(\Du(1)=0\); 
\item \(\mathbb K[x]\) is invariant under \(\Du\). 
\end{enumerate} 
\end{proposition} 
\begin{proof} 
The equivalence between (1) and (2) follows immediately from \eqref{eq:Du-on-constant-recall}. The equivalence between (1) and (3) was established in Corollary~\ref{cor:polynomial-invariance}. Indeed, \[ \Du(\mathbb K[x]) \subseteq \mathbb K[x] \] holds precisely when \(\U_0=0. \) 
\end{proof} 

From the operator-theoretic point of view, the regimes may be summarized as follows: 
\[ 
\U_0=0 \quad\Longleftrightarrow\quad \Du(1)=0 \quad\Longleftrightarrow\quad \Du(\mathbb K[x]) \subseteq \mathbb K[x], 
\] 
whereas 
\[ 
\U_0\neq0 \quad\Longleftrightarrow\quad \Du(1)=\U_0x^{-1} \quad\Longrightarrow\quad \Du(\mathbb K[x]) \subseteq \mathbb K[x,x^{-1}]. \] 
This dichotomy determines whether degree zero acts as a terminal level of the lowering dynamics or as a transition point toward negative powers.

\section{Algebraic structure and iteration}\label{sec:algebraic-structure-iteration}

The fundamental monomial identity established in the preceding section,
\begin{equation}
\label{eq:fundamental-identity-iteration}
\Du x^n
=
\U_nx^{n-1},
\qquad n\geq0,
\end{equation}
Already determines a substantial part of the algebraic structure of
$\Du$. In particular, repeated application of the operator produces
products of consecutive recurrence values. This section develops the
resulting iteration theory and identifies the algebraic mechanism by which vanishing recurrence values interrupt the lowering process.

\subsection{Iterates on monomials}

The first few iterates of a monomial illustrate the general pattern. From
\eqref{eq:fundamental-identity-iteration},
\begin{align}
\Du x^n
&=
\U_nx^{n-1},\\
\Du^2x^n
&=
\U_n\U_{n-1}x^{n-2},\\
\Du^3x^n
&=
\U_n\U_{n-1}\U_{n-2}x^{n-3}.
\end{align}
The general formula follows by induction.
\begin{theorem}[Iterates on monomials]
\label{thm:iterates-monomials}
Let $n\geq 0$, and let $m$ be an integer such that
$
0\leq m\leq n.
$
Then
\begin{equation}
\label{eq:iterates-monomials}
\Du^{m}x^n
=
\left(
\prod_{j=0}^{m-1}\U_{n-j}
\right)x^{n-m}.
\end{equation}
For $m=0$, the product in \eqref{eq:iterates-monomials} is understood
to be the empty product and hence equal to $1$.
\end{theorem}

\begin{proof}
We proceed by induction on $m$. For $m=0$, the empty-product convention gives
\begin{equation}
\Du^{0}x^n
=
x^n
=
\left(
\prod_{j=0}^{-1}\U_{n-j}
\right)x^n,
\end{equation}
so \eqref{eq:iterates-monomials} holds. For $m=1$, the formula reduces to $\Du x^n=\U_nx^{n-1}$, which is precisely the fundamental monomial identity. Now suppose that, for some integer $m$ with $1\leq m<n$,
\begin{equation}
\label{eq:iteration-induction-hypothesis}
\Du^{m}x^n
=
\left(
\prod_{j=0}^{m-1}\U_{n-j}
\right)x^{n-m}.
\end{equation}
Applying $\Du$ to both sides of
\eqref{eq:iteration-induction-hypothesis} and using linearity, we obtain
\begin{align*}
\Du^{m+1}x^n
=
\Du
\left[
\left(
\prod_{j=0}^{m-1}\U_{n-j}
\right)x^{n-m}
\right]
=
\left(
\prod_{j=0}^{m-1}\U_{n-j}
\right)
\Du x^{n-m}.
\end{align*}
Since $m<n$, we have $n-m\geq1$, and the fundamental monomial
identity gives
\begin{equation}
\Du x^{n-m}
=
\U_{n-m}x^{n-m-1}.
\end{equation}
Therefore,
\begin{align*}
\Du^{m+1}x^n
=
\left(
\prod_{j=0}^{m-1}\U_{n-j}
\right)
\U_{n-m}x^{n-m-1}
=
\left(
\prod_{j=0}^{m}\U_{n-j}
\right)
x^{n-(m+1)}.
\end{align*}
This is exactly \eqref{eq:iterates-monomials} with $m$ replaced by
$m+1$. Hence the result follows by induction for every
$0\leq m\leq n$.
\end{proof}

\begin{corollary}[Complete monomial descent]\label{cor:complete-monomial-descent}
For every $n\geq1$, 
\begin{equation} 
\Du^n x^n = \U_n\U_{n-1}\cdots\U_1. 
\end{equation} 
\end{corollary}

\begin{corollary}[Termination degree zero]\label{cor:termination-degree-zero}
Suppose $\U_0=0.$ Then, for every $n\geq0$, 
$\Du^{n+1}x^n=0.$ Consequently, $\Du^m x^n=0,$ for $m>n. $ 
\end{corollary}

\subsection{Iterates on polynomials and local nilpotence}

The monomial iteration formula extends naturally to arbitrary 
polynomials. In the Ward regime, the operator is locally nilpotent 
and its iterates admit an explicit description in terms of consecutive values of the recurrence sequence. 

\begin{proposition}[Iterates on polynomials] \label{prop:iterates-polynomials} 
Assume that $\U_0=0$. Let $f(x) = \sum_{n=0}^{N}c_nx^n$. 
Then, for every integer \(m\) satisfying \( 0\leq m\leq N, \) we have 
\begin{equation} \label{eq:iterates-polynomials} 
\Du^{m}f(x) = \sum_{n=m}^{N} c_n \left( \prod_{j=0}^{m-1}\U_{n-j} \right) x^{n-m}. 
\end{equation} 
For \(m=0\), the product is understood to be empty and equal to \(1\). 
\end{proposition} 

The preceding formula shows that each application of \(\Du\) lowers the degree of every surviving monomial component by one. 

\begin{corollary}[Degree-lowering property] 
\label{cor:degree-lowering} 
Assume that \(\U_0=0\), and let \( f\in\mathbb K[x] \) satisfy \( \deg f\leq N. \) Then, for every \(m\) with \( 0\leq m\leq N, \) \begin{equation} \label{eq:degree-lowering} \deg\bigl(\Du^{m}f\bigr) \leq N-m, \end{equation} whenever \(\Du^{m}f\neq0\).
\end{corollary} 

For each integer \(N\geq 0\), define 
\begin{equation} \label{eq:PN-definition} 
\mathcal P_N = \operatorname{span}_{\mathbb K} \{1,x,\ldots,x^N\}. 
\end{equation} 
We now obtain the principal finiteness property of the Ward regime. 

\begin{theorem}[Local nilpotence] \label{thm:local-nilpotence} 
Assume that $\U_0=0$. Then \(\Du\) is locally nilpotent on \(\mathbb K[x]\). More precisely, for every \(N\geq0\), 
\begin{equation} \label{eq:nilpotence-on-PN} 
\Du^{N+1}\big|_{\mathcal P_N} = 0. \end{equation} Equivalently, if \(f\in\mathbb K[x] \) satisfies \( \deg f\leq N, \) then 
\begin{equation} \label{eq:local-nilpotence} 
\Du^{N+1}f = 0. 
\end{equation} 
\end{theorem} 
\begin{proof} 
Let \[ f(x) = \sum_{n=0}^{N}c_nx^n. \] For every monomial appearing in \(f\), Corollary~\ref{cor:termination-degree-zero} gives \( \Du^{n+1}x^n=0. \) Since \(N+1\geq n+1, \) it follows that \( \Du^{N+1}x^n=0\), \(0\leq n\leq N. \) Therefore 
\begin{align} 
\Du^{N+1}f &= \sum_{n=0}^{N} c_n \Du^{N+1}x^n= 0. 
\end{align} Hence \[ \Du^{N+1}\big|_{\mathcal P_N} = 0. \] Since every polynomial belongs to some finite-dimensional space \(\mathcal P_N\), the operator \(\Du\) is locally nilpotent on \(\mathbb K[x]\). \end{proof}

\section{Spectral product rule}\label{sec:product-rule}

\subsection{Spectral increments}

The \(U\)-derivative is defined through simultaneous evaluation at the characteristic scales \[ \lambda_1 x,\ldots,\lambda_k x. \] Consequently, its interaction with multiplication is controlled by the entire characteristic spectrum rather than by evaluation at a single point. Throughout this section, let \[ r,s\in\{1,\ldots,k\}, \qquad r\neq s. \] The characteristic roots \(\lambda_r\) and \(\lambda_s\) will serve as reference roots for the two factors appearing in a product. 

\begin{definition}[Spectral increments] \label{def:spectral-increments} 
For every index \(j\in\{1,\ldots,k\}\), define the spectral increment of a function \(h\) relative to \(\lambda_r\) by 
\begin{equation} \label{eq:spectral-increment-f} 
\Delta_{j,r}h = T_{\lambda_j}h - T_{\lambda_r}h. 
\end{equation} 
\end{definition} 

The increments measure the deviation of a given spectral component from the corresponding reference component. \begin{remark} \label{rem:vanishing-spectral-increments} Two increments vanish identically: \begin{equation} \label{eq:vanishing-increments} \Delta_{r,r}f=0, \qquad \Delta_{s,s}g=0. \end{equation} \end{remark} Since the indices \(r\) and \(s\) play a distinguished role, it is convenient to isolate the remaining characteristic channels. \begin{definition} \label{def:Jrs} For distinct indices \(r\) and \(s\), define \begin{equation} \label{eq:Jrs} J_{r,s} = \{1,\ldots,k\} \setminus \{r,s\}. \end{equation} \end{definition} Since \(r\neq s\), \begin{equation} \label{eq:Jrs-cardinality} |J_{r,s}|=k-2. \end{equation} The set \(J_{r,s}\) records precisely those characteristic modes that remain after the two reference roots have been selected. As we shall see in the next subsection, these remaining modes are responsible for the higher-order spectral interactions appearing in the product rule.

\subsection{The General Spectral Product Rule}

We now derive the fundamental multiplicative identity satisfied by the \(U\)-derivative.

\begin{theorem}[Spectral product rule] \label{thm:spectral-product-rule} 
Let \[ r,s\in\{1,\ldots,k\}, \qquad r\neq s. \] Then, for every admissible pair of functions \(f\) and \(g\), 
\begin{align} \label{eq:spectral-product-rule} 
\Du(fg) ={}& T_{\lambda_r}f\,\Du g + T_{\lambda_s}g\,\Du f + \frac1x \sum_{j\in J_{r,s}} A_j \Delta_{j,r}f \, \Delta_{j,s}g - \frac{\U_0}{x} T_{\lambda_r}f T_{\lambda_s}g.
\end{align} 
\end{theorem} 
\begin{proof} 
Define 
\begin{equation} \label{eq:auxiliary-sum-product-rule} 
S_{r,s}(f,g) = \sum_{j=1}^{k} A_j \Delta_{j,r}f \Delta_{j,s}g . 
\end{equation} 
Using the definitions of the spectral increments, 
\begin{align} 
S_{r,s}(f,g) = \sum_{j=1}^{k} A_j \bigl( T_{\lambda_j}f - T_{\lambda_r}f \bigr) \bigl( T_{\lambda_j}g - T_{\lambda_s}g \bigr). \label{eq:expanded-product-rule} 
\end{align} 
Expanding the product gives 
\begin{align} 
S_{r,s}(f,g) ={}& \sum_{j=1}^{k} A_j T_{\lambda_j}f T_{\lambda_j}g  - T_{\lambda_s}g \sum_{j=1}^{k} A_j T_{\lambda_j}f  - T_{\lambda_r}f \sum_{j=1}^{k} A_j T_{\lambda_j}g + T_{\lambda_r}f T_{\lambda_s}g \sum_{j=1}^{k} A_j . \label{eq:expanded-Srs} 
\end{align} 
Since $\sum_{j=1}^{k} A_jT_{\lambda_j}h = x\,\Du h, $ for every admissible function \(h\), and $ \sum_{j=1}^{k} A_j = \U_0, $ equation \eqref{eq:expanded-Srs} becomes \begin{align} 
S_{r,s}(f,g) = x\Du(fg) - xT_{\lambda_s}g\,\Du f - xT_{\lambda_r}f\,\Du g + \U_0 T_{\lambda_r}f T_{\lambda_s}g . \end{align} Solving for \(\Du(fg)\) yields 
\begin{align} 
\Du(fg) = T_{\lambda_r}f\,\Du g + T_{\lambda_s}g\,\Du f + \frac1x S_{r,s}(f,g) - \frac{\U_0}{x} T_{\lambda_r}f T_{\lambda_s}g . \label{eq:before-Jrs-reduction} 
\end{align} 
Now observe that \[ \Delta_{r,r}f=0, \qquad \Delta_{s,s}g=0. \] Hence the terms corresponding to \(j=r\) and \(j=s\) vanish identically in \eqref{eq:auxiliary-sum-product-rule}. Therefore \[ S_{r,s}(f,g) = \sum_{j\in J_{r,s}} A_j \Delta_{j,r}f \Delta_{j,s}g . \] Substituting this expression into \eqref{eq:before-Jrs-reduction} gives \eqref{eq:spectral-product-rule}. 
\end{proof}

\begin{proposition}[Ward form of the product rule] \label{prop:ward-product-rule} 
Assume that \(\U_0=0. \) Then 
\begin{align} \label{eq:ward-product-rule} 
\Du(fg) ={}& T_{\lambda_r}f\,\Du g + T_{\lambda_s}g\,\Du f + \frac1x \sum_{j\in J_{r,s}} A_j \Delta_{j,r}f \Delta_{j,s}g . 
\end{align} 
\end{proposition} 
\begin{proof} The result follows immediately from Theorem~\ref{thm:spectral-product-rule}, since the term \[ - \frac{\U_0}{x} T_{\lambda_r}f T_{\lambda_s}g \] vanishes identically when \(\U_0=0\). 
\end{proof} 
The non-Ward form of the product rule is exactly Theorem~\ref{thm:spectral-product-rule}. The non-Ward regime contains two contributions beyond the first-order terms. The first is the residual spectral interaction \[ \frac1x \sum_{j\in J_{r,s}} A_j \Delta_{j,r}f \Delta_{j,s}g, \] while the second is the explicit initial-value contribution \[ - \frac{\U_0}{x} T_{\lambda_r}f T_{\lambda_s}g. \] Consequently, the Ward and non-Ward regimes differ not only in the action of \(\Du\) on constants, but also in the structure of multiplication.

The spectral product rule will be analyzed in the following subsections. Its structure depends on three ingredients: the two reference spectral modes, the remaining characteristic channels, and the distinguished initial value \(\U_0\).

\subsection{Structure of the Product Rule}

The spectral product rule reveals that the multiplicative structure of \(\Du\) is governed by three distinct ingredients: 
\begin{enumerate} 
\item two first-order spectral contributions; 
\item interactions among the remaining characteristic modes; 
\item an initial-value contribution determined by \(\U_0\). 
\end{enumerate} 
These components originate from different pieces of the spectral data and play different roles in the subsequent theory. 
\begin{corollary}[Structural decomposition]\label{cor:product-rule-decomposition} 
The product rule \eqref{eq:spectral-product-rule} decomposes uniquely as 
\begin{equation} \label{eq:product-rule-decomposition} 
\Du(fg) = L_{r,s}(f,g) + I_{r,s}(f,g) + C_{r,s}(f,g), 
\end{equation} where 
\begin{align} 
L_{r,s}(f,g) &= T_{\lambda_r}f\,\Du g + T_{\lambda_s}g\,\Du f, \\ I_{r,s}(f,g) &= \frac1x \sum_{j\in J_{r,s}} A_j \Delta_{j,r}f \Delta_{j,s}g, \\ C_{r,s}(f,g) &= - \frac{\U_0}{x} T_{\lambda_r}f T_{\lambda_s}g. 
\end{align} 
\end{corollary} 
\begin{proof} 
This is merely a regrouping of the three terms appearing in Theorem~\ref{thm:spectral-product-rule}. 
\end{proof} 
The first contribution, \[ L_{r,s}(f,g), \] plays the role of a first-order Leibniz-type expression. The remaining two terms measure the deviation of \(\Du\) from a two-term product rule. The interaction term \[ I_{r,s}(f,g) = \frac1x \sum_{j\in J_{r,s}} A_j \Delta_{j,r}f \Delta_{j,s}g \] contains all dependence on the characteristic modes not chosen as reference roots. These modes form the residual spectral subsystem associated with the product. 
\begin{proposition}[Number of interaction channels] \label{prop:number-spectral-channels} 
For a recurrence sequence of order \(k\), the interaction term contains exactly $k-2$ spectral channels. 
\end{proposition} \begin{proof} 
By Definition~\ref{def:Jrs}, 
\[ 
J_{r,s} = \{1,\ldots,k\} \setminus \{r,s\}. 
\] 
Since the indices \(r\) and \(s\) are distinct, \(|J_{r,s}| = k-2. \) The interaction term contains precisely one summand for each element of \(J_{r,s}\). Hence the number of spectral interaction channels is \(k-2\). \end{proof} The proposition gives a simple hierarchy: \begin{center} \begin{tabular}{c|c} Order \(k\) & Number of interaction channels\\ \hline 2 & 0\\ 3 & 1\\ 4 & 2\\ 5 & 3\\ \(\vdots\) & \(\vdots\)\\ \(k\) & \(k-2\) \end{tabular} \end{center} 

The next subsection examines the first low-order cases. There the number of spectral interaction channels is sufficiently small that the general product rule simplifies considerably.

\subsection{Low-Order Cases}

The general product rule becomes particularly transparent in low orders. Since the interaction term contains exactly \(k-2\) spectral channels, the complexity of the multiplicative structure increases with the order of the recurrence. 

\begin{corollary}[Second-order product rule] \label{cor:order-two-product-rule} 
Assume that \(k=2\). Let the characteristic roots be \(\lambda_1,\lambda_2, \) and choose \( \{r,s\} = \{1,2\}. \) Then \begin{equation} \label{eq:order-two-product-rule} \Du(fg) = T_{\lambda_r}f\,\Du g + T_{\lambda_s}g\,\Du f - \frac{\U_0}{x} T_{\lambda_r}f T_{\lambda_s}g. 
\end{equation} 
\end{corollary} 
\begin{proof} 
If \(k=2\), then \(J_{r,s} = \varnothing. \) Therefore the interaction term in Theorem~\ref{thm:spectral-product-rule} vanishes identically. The result follows immediately. 
\end{proof} 
Thus second-order recurrences are exceptional: after choosing the two reference roots, no additional spectral channel remains. 

\begin{corollary}[Second-order Ward form] \label{cor:order-two-ward} Assume that \(k=2\) and \(\U_0=0. \) Then 
\begin{equation} \label{eq:order-two-ward} 
\Du(fg) = T_{\lambda_r}f\,\Du g + T_{\lambda_s}g\,\Du f. \end{equation} 
\end{corollary}
Equation~\eqref{eq:order-two-ward} is the closest analogue of a Leibniz rule in the distinct-root theory. The result contains exactly two first-order terms and no higher spectral correction. The situation changes immediately in order three. 

\begin{corollary}[Third-order product rule] \label{cor:order-three-product-rule} 
Assume that \(k=3\). Let \(\{r,s,\ell\} = \{1,2,3\}. \) Then \begin{align} \label{eq:order-three-product-rule} 
\Du(fg) ={}& T_{\lambda_r}f\,\Du g + T_{\lambda_s}g\,\Du f  + \frac{A_{\ell}}{x} \Delta_{\ell,r}f \Delta_{\ell,s}g - \frac{\U_0}{x} T_{\lambda_r}f T_{\lambda_s}g. 
\end{align} 
\end{corollary} 
\begin{proof} 
If \(k=3\), then \(J_{r,s} = \{\ell\}. \) Hence the interaction sum appearing in Theorem~\ref{thm:spectral-product-rule} contains exactly one term: \[ \sum_{j\in J_{r,s}} A_j \Delta_{j,r}f \Delta_{j,s}g = A_{\ell} \Delta_{\ell,r}f \Delta_{\ell,s}g. \] Substituting into \eqref{eq:spectral-product-rule} gives \eqref{eq:order-three-product-rule}. 
\end{proof} 

\begin{corollary}[Fourth-order product rule] \label{cor:order-four-product-rule} 
Assume that \(k=4\). Let \(J_{r,s}=\{p,q\}. \) Then 
\begin{align} 
\Du(fg) ={}& T_{\lambda_r}f\,\Du g + T_{\lambda_s}g\,\Du f + \frac1x \Big( A_p\Delta_{p,r}f\Delta_{p,s}g + A_q\Delta_{q,r}f\Delta_{q,s}g \Big)\\
& - \frac{\U_0}{x} T_{\lambda_r}f T_{\lambda_s}g . \nonumber 
\end{align} 
\end{corollary}

\begin{remark} 
Order four is the first order for which distinct interaction channels may cancel each other. This phenomenon does not occur in orders two and three. 
\end{remark}

The low-order hierarchy may therefore be summarized as 
\begin{center} 
\begin{tabular}{c|c|c} 
Order & Interaction channels & New phenomenon\\ \hline 2 & \(0\) & no spectral interaction\\ 3 & \(1\) & one interaction channel\\ 4 & \(2\) & interaction-channel cancellation 
\end{tabular} 
\end{center} 
Beginning with order four, the interaction term acquires an internal geometry. 

\subsection{Monomial Consequences}

The spectral product rule induces nontrivial identities for the recurrence sequence itself. These identities are obtained by evaluating the product rule on monomials. 

\begin{proposition}[Monomial product identity] \label{prop:monomial-product-identity} 
For every \(m,n\geq0\), 
\begin{align} \label{eq:monomial-product-identity} 
\U_{m+n} ={}& \lambda_r^{m}\U_n + \lambda_s^{n}\U_m - \U_0 \lambda_r^{m} \lambda_s^{n} + \sum_{j\in J_{r,s}} A_j \bigl( \lambda_j^{m} - \lambda_r^{m} \bigr) \bigl( \lambda_j^{n} - \lambda_s^{n} \bigr).
\end{align} 
\end{proposition} 
\begin{proof} 
Let \(f(x)=x^m\), \(g(x)=x^n. \) Then \(fg=x^{m+n}. \) By the fundamental monomial identity, \[ \Du(fg) = \U_{m+n}x^{m+n-1}. \] Moreover, 
\[ T_{\lambda_r}f = \lambda_r^m x^m, \qquad T_{\lambda_s}g = \lambda_s^n x^n, \] 
and 
\[ \Du(x^m) = \U_mx^{m-1}, \qquad \Du(x^n) = \U_nx^{n-1}. \] 
Substituting these expressions into Theorem~\ref{thm:spectral-product-rule} gives 
\begin{align} 
\U_{m+n}x^{m+n-1} ={}& \lambda_r^m\U_nx^{m+n-1} + \lambda_s^n\U_mx^{m+n-1} - \U_0 \lambda_r^m \lambda_s^n x^{m+n-1}\\ & + \sum_{j\in J_{r,s}} A_j (\lambda_j^m-\lambda_r^m) (\lambda_j^n-\lambda_s^n) x^{m+n-1}. \nonumber \end{align} Dividing by \(x^{m+n-1}\) yields \eqref{eq:monomial-product-identity}. 
\end{proof} 
The identity shows that the functional product rule is reflected by a corresponding addition law for the recurrence sequence. The complexity of that addition law increases with the number of spectral interaction channels. 

\begin{corollary}[Second-order addition identity] \label{cor:order-two-addition} 
Assume that \(k=2\). Then 
\begin{equation} \label{eq:order-two-addition} \U_{m+n} = \lambda_r^m\U_n + \lambda_s^n\U_m - \U_0 \lambda_r^m \lambda_s^n . \end{equation} 
\end{corollary} 

\begin{corollary}[Second-order Ward addition identity] \label{cor:order-two-ward-addition} 
Assume that \(k=2, \U_0=0. \) Then 
\begin{equation} \label{eq:order-two-ward-addition} 
\U_{m+n} = \lambda_r^m\U_n + \lambda_s^n\U_m. 
\end{equation} 
\end{corollary} 

Thus the second-order Ward regime possesses a particularly simple addition law. Beginning with order three, higher spectral corrections appear, and from order four onward several interaction channels may contribute simultaneously. 
\begin{remark} 
The identity \eqref{eq:monomial-product-identity} reveals a duality between the operator and the recurrence sequence. The spectral product rule determines an addition law for the recurrence coefficients, while the recurrence coefficients determine the monomial action of the operator through \[ \Du x^n=\U_nx^{n-1}. \] The operator and the recurrence therefore encode the same spectral information in two different forms. 
\end{remark} 
The spectral product rule has now been completely analyzed. The next section studies its deviation from a two-term spectral Leibniz identity. This deviation is measured by the Leibniz defect, which provides the natural criterion for determining when the higher spectral interactions cancel.

\section{Leibniz defect}\label{sec:leibniz-defect}

\subsection{Definition of the Leibniz Defect}

The spectral product rule derived in the previous section shows that the \(U\)-derivative does not, in general, satisfy a two-term Leibniz identity. The additional terms arise from two independent sources: 
\begin{enumerate} 
\item the interaction among the characteristic modes that remain after the reference roots have been selected; 
\item the contribution generated by the distinguished recurrence value \(\U_0\). 
\end{enumerate} 
It is therefore natural to isolate the difference between \(\Du(fg)\) and the two first-order spectral contributions. 
\begin{definition}[Leibniz defect] \label{def:leibniz-defect} 
Let \[ r,s\in\{1,\ldots,k\}, \qquad r\neq s. \] The \((r,s)\)-Leibniz defect of an admissible pair of functions \((f,g)\) is defined by 
\begin{equation} \label{eq:leibniz-defect} 
\mathcal E_{r,s}(f,g) = \Du(fg) - T_{\lambda_r}f\,\Du g - T_{\lambda_s}g\,\Du f. \end{equation} 
\end{definition} 
Thus \[ \mathcal E_{r,s}(f,g) \] measures the failure of the \(U\)-derivative to satisfy the two-term spectral Leibniz identity associated with the reference roots \(\lambda_r\) and \(\lambda_s\). Using the spectral product rule, the defect admits an explicit representation. 
\begin{proposition}[Explicit defect formula] \label{prop:defect-expression} 
For every admissible pair of functions \(f\) and \(g\), 
\begin{align} \label{eq:defect-expression} 
\mathcal E_{r,s}(f,g) = {}& \frac1x \sum_{j\in J_{r,s}} A_j \Delta_{j,r}f \Delta_{j,s}g - \frac{\U_0}{x} T_{\lambda_r}f T_{\lambda_s}g . 
\end{align} 
\end{proposition} 
\begin{proof} Subtracting \[ T_{\lambda_r}f\,\Du g + T_{\lambda_s}g\,\Du f \] from both sides of Theorem~\ref{thm:spectral-product-rule} gives \eqref{eq:defect-expression}. \end{proof} The defect separates naturally into two contributions: \begin{equation} \label{eq:defect-decomposition} \mathcal E_{r,s} = \mathcal S_{r,s} + \mathcal C_{r,s}, \end{equation} where \begin{equation} \label{eq:spectral-defect-part} \mathcal S_{r,s}(f,g) = \frac1x \sum_{j\in J_{r,s}} A_j \Delta_{j,r}f \Delta_{j,s}g, \end{equation} and \begin{equation} \label{eq:initial-value-defect-part} \mathcal C_{r,s}(f,g) = - \frac{\U_0}{x} T_{\lambda_r}f T_{\lambda_s}g. \end{equation} The first term records the interaction among the residual characteristic channels, whereas the second depends only on the distinguished recurrence value \(\U_0\). \begin{remark} The Leibniz defect is identically zero if and only if the spectral product rule reduces to the two-term identity \[ \Du(fg) = T_{\lambda_r}f\,\Du g + T_{\lambda_s}g\,\Du f. \] Consequently, the study of the multiplicative structure of \(\Du\) is equivalent to the study of the corresponding defect. \end{remark} The explicit formula \eqref{eq:defect-expression} shows that vanishing defect is governed by an equation relating spectral interaction terms and the contribution generated by \(\U_0\). The next subsection establishes the resulting general zero-defect criterion.

\subsection{The General Zero-Defect Criterion} 
The explicit formula obtained in the preceding subsection shows that the Leibniz defect is completely determined by the interaction between the remaining spectral channels and the contribution generated by the distinguished recurrence value \(\U_0\). We now characterize precisely when the defect vanishes. 
\begin{theorem}[General zero-defect criterion] \label{thm:zero-defect-criterion} 
Let \(r\neq s\). For an admissible pair of functions \((f,g)\), the following statements are equivalent: 
\begin{enumerate} 
\item \[ \mathcal E_{r,s}(f,g)=0; \] 
\item 
\begin{equation} \label{eq:general-zero-defect-condition} 
\sum_{j\in J_{r,s}} A_j \Delta_{j,r}f \Delta_{j,s}g = \U_0 T_{\lambda_r}f T_{\lambda_s}g. 
\end{equation} 
\end{enumerate} 
\end{theorem} 
\begin{proof} 
By Proposition~\ref{prop:defect-expression}, \[ \mathcal E_{r,s}(f,g) = \frac1x \sum_{j\in J_{r,s}} A_j \Delta_{j,r}f \Delta_{j,s}g - \frac{\U_0}{x} T_{\lambda_r}f T_{\lambda_s}g. \] Therefore \[ \mathcal E_{r,s}(f,g)=0 \] if and only if \[ \frac1x \sum_{j\in J_{r,s}} A_j \Delta_{j,r}f \Delta_{j,s}g = \frac{\U_0}{x} T_{\lambda_r}f T_{\lambda_s}g. \] Multiplication by \(x\) yields \eqref{eq:general-zero-defect-condition}. The converse follows by reversing the argument. 
\end{proof} 
The theorem shows that zero defect is governed by a balance equation between two quantities of different origin. The left-hand side of \eqref{eq:general-zero-defect-condition} depends only on the residual spectral channels contained in \(J_{r,s}\). In contrast, the right-hand side depends only on the distinguished recurrence value \(\U_0\) and the two reference roots. Consequently, the vanishing of the defect is not a purely local property of the functions involved. It reflects a compatibility relation between the characteristic spectrum and the initial-value structure of the recurrence sequence. 
\begin{corollary} \label{cor:zero-defect-equation} 
The pair \((f,g)\) satisfies the two-term spectral Leibniz identity 
\begin{equation} \label{eq:two-term-spectral-leibniz} 
\Du(fg) = T_{\lambda_r}f\,\Du g + T_{\lambda_s}g\,\Du f 
\end{equation} 
if and only if \eqref{eq:general-zero-defect-condition} holds. 
\end{corollary} 
\begin{proof} 
By Definition~\ref{def:leibniz-defect}, equation~\eqref{eq:two-term-spectral-leibniz} is equivalent to \[ \mathcal E_{r,s}(f,g)=0. \] The conclusion therefore follows immediately from Theorem~\ref{thm:zero-defect-criterion}. 
\end{proof} 
The criterion \eqref{eq:general-zero-defect-condition} contains the two principal regimes of the theory as special cases. If \[ \U_0=0, \] the right-hand side vanishes and the condition reduces to a pure spectral cancellation equation. If \[ \U_0\neq0, \] the spectral interaction must instead balance the explicit contribution generated by the initial value. These two situations will be studied separately in the next subsections.

\subsection{The Ward Regime: Spectral Cancellation}

We now specialize the general zero-defect criterion to the Ward regime.  
\begin{proposition}[Ward defect formula] \label{prop:ward-defect} 
Assume that \(\U_0=0. \) Then 
\begin{equation} \label{eq:ward-defect} 
\mathcal E_{r,s}(f,g) = \frac1x \sum_{j\in J_{r,s}} A_j \Delta_{j,r}f \Delta_{j,s}g. 
\end{equation} 
\end{proposition} 
\begin{proof} 
By Proposition~\ref{prop:defect-expression}, \[ \mathcal E_{r,s}(f,g) = \frac1x \sum_{j\in J_{r,s}} A_j \Delta_{j,r}f \Delta_{j,s}g - \frac{\U_0}{x} T_{\lambda_r}f T_{\lambda_s}g. \] 
Since \(\U_0=0, \) the second term vanishes identically, yielding \eqref{eq:ward-defect}. 
\end{proof} 
The defect is therefore determined entirely by the residual spectral interaction channels. 
\begin{corollary}[Ward zero-defect criterion] \label{cor:ward-zero-defect} 
Assume that \(\U_0=0. \) Then \( \mathcal E_{r,s}(f,g)=0 \) if and only if 
\begin{equation} \label{eq:ward-zero-defect-condition} 
\sum_{j\in J_{r,s}} A_j \Delta_{j,r}f \Delta_{j,s}g = 0. 
\end{equation} 
\end{corollary} 
\begin{proof} 
Combining Theorem~\ref{thm:zero-defect-criterion} with the identity \(\U_0=0 \) gives \eqref{eq:ward-zero-defect-condition}. 
\end{proof} 
The Ward regime is therefore governed by a pure cancellation phenomenon: the interaction channels must annihilate each other exactly. Beginning with order four, several nonzero interaction channels may cancel each other. To describe this phenomenon more geometrically, define 
\begin{equation} \label{eq:Vr-definition} 
V_r f(x) = \bigl( \Delta_{j,r}f(x) \bigr)_{j\in J_{r,s}}, 
\end{equation} 
and 
\begin{equation} \label{eq:Vs-definition} 
V_s g(x) = \bigl( \Delta_{j,s}g(x) \bigr)_{j\in J_{r,s}}. 
\end{equation} 
On \(\mathbb K^{k-2}\), introduce the diagonal bilinear form 
\begin{equation} \label{eq:weighted-bilinear-form} 
\langle u,v\rangle_A = \sum_{j\in J_{r,s}} A_j u_j v_j. \end{equation} 

\begin{proposition}[Spectral orthogonality formula] \label{prop:spectral-orthogonality-formula} 
Assume that \(\U_0=0. \) Then 
\begin{equation} \label{eq:defect-orthogonality} 
\mathcal E_{r,s}(f,g) = \frac1x \langle V_r f, V_s g \rangle_A. 
\end{equation} 
\end{proposition} 
\begin{proof} 
Substituting \eqref{eq:Vr-definition}, \eqref{eq:Vs-definition}, and \eqref{eq:weighted-bilinear-form} into \eqref{eq:ward-defect} yields \eqref{eq:defect-orthogonality}. \end{proof} 

\begin{corollary}[Spectral orthogonality criterion] \label{cor:spectral-orthogonality-criterion} 
Assume that \(\U_0=0. \) Then \(\mathcal E_{r,s}(f,g)=0 \) if and only if 
\begin{equation} \label{eq:spectral-orthogonality} 
\langle V_r f, V_s g \rangle_A = 0. 
\end{equation} 
\end{corollary} 
\begin{proof} By Proposition~\ref{prop:spectral-orthogonality-formula}, \[ \mathcal E_{r,s}(f,g) = \frac1x \langle V_r f, V_s g \rangle_A. \] Thus the defect vanishes exactly when the weighted inner product is zero. 
\end{proof} 
The Ward regime is therefore governed entirely by the geometry of the residual spectral channels. The vanishing of the defect corresponds to spectral cancellation, which may arise either from invariance (\(V_r f=0\) or \(V_s g=0\)) or from genuine orthogonality among nontrivial interaction channels. The non-Ward regime behaves differently. There the interaction channels do not need to cancel completely; instead they must balance the explicit contribution generated by the nonzero value \(\U_0\). This situation is considered in the next subsection.

\subsection{The Non-Ward Regime: Spectral Balance}

We now consider the complementary situation \(\U_0\neq0\). In contrast with the Ward regime, the Leibniz defect contains an explicit initial-value contribution. Consequently, vanishing defect is no longer equivalent to the cancellation of the residual spectral channels. 
\begin{proposition}[Non-Ward defect formula] \label{prop:nonward-defect} 
Assume that \(\U_0\neq0. \) Then 
\begin{equation} \label{eq:nonward-defect} 
\mathcal E_{r,s}(f,g) = \frac1x \sum_{j\in J_{r,s}} A_j \Delta_{j,r}f \Delta_{j,s}g - \frac{\U_0}{x} T_{\lambda_r}f T_{\lambda_s}g. \end{equation} 
\end{proposition} 
\begin{proof} 
The identity follows directly from Proposition~\ref{prop:defect-expression} under the assumption \(\U_0\neq0. \) 
\end{proof} 
The defect therefore contains two independent contributions: \begin{enumerate} 
\item the residual spectral interaction \[ \sum_{j\in J_{r,s}} A_j \Delta_{j,r}f \Delta_{j,s}g; \] 
\item the initial-value contribution \[ \U_0 T_{\lambda_r}f T_{\lambda_s}g. \] 
\end{enumerate} 
Consequently, defect-free behavior requires an exact balance between these two quantities. 
\begin{corollary}[Non-Ward zero-defect criterion] \label{cor:nonward-zero-defect} 
Assume that \(\U_0\neq0. \) Then 
\(\mathcal E_{r,s}(f,g)=0 \) if and only if 
\begin{equation} \label{eq:nonward-balance} 
\sum_{j\in J_{r,s}} A_j \Delta_{j,r}f \Delta_{j,s}g = \U_0 T_{\lambda_r}f T_{\lambda_s}g. 
\end{equation} 
\end{corollary} 
\begin{proof} This is the specialization of Theorem~\ref{thm:zero-defect-criterion} to the case \(\U_0\neq0. \) 
\end{proof} Equation~\eqref{eq:nonward-balance} should be compared with the Ward condition \[ \sum_{j\in J_{r,s}} A_j \Delta_{j,r}f \Delta_{j,s}g = 0. \] In the Ward regime, the residual spectral channels must cancel each other. In the non-Ward regime, the same channels must instead reproduce exactly the contribution generated by the distinguished value \(\U_0\). \begin{remark}[Spectral balance] \label{rem:spectral-balance} The vanishing of the Leibniz defect is governed by two fundamentally different mechanisms: \[ \text{Ward} \quad\Longrightarrow\quad \text{spectral cancellation}, \] \[ \text{non-Ward} \quad\Longrightarrow\quad \text{spectral balance}. \] The first requires the residual interaction channels to annihilate each other. The second requires those channels to compensate exactly for the initial-value contribution produced by \(\U_0\). \end{remark} Thus the same bifurcation already detected through \[ \Du(1)=0 \qquad\text{or}\qquad \Du(1)=\U_0x^{-1} \] reappears in the multiplicative structure of the operator. In the Ward regime, zero defect is determined entirely by spectral geometry. In the non-Ward regime, it is determined by a balance between spectral geometry and the distinguished initial value. The next subsection examines this phenomenon on monomials and derives the corresponding arithmetic characterization in terms of the recurrence sequence itself.

\section{Quotient Rule} 

The quotient rule introduces no new structural obstruction beyond the Leibniz defect. Let \[ h=\frac{f}{g}, \] and assume that the spectral evaluations \[ T_{\lambda_r}g,\qquad T_{\lambda_s}g \] are nonzero on the domain under consideration. Since \[ f=hg, \] the spectral product rule gives \[ D_Uf = T_{\lambda_r}h\,D_Ug + T_{\lambda_s}g\,D_Uh + \mathcal L_{r,s}(h,g), \] where \[ \mathcal L_{r,s}(h,g) = D_U(hg) - T_{\lambda_r}h\,D_Ug - T_{\lambda_s}g\,D_Uh \] is the $(r,s)$-Leibniz defect. Using \[ T_{\lambda_r}h = \frac{T_{\lambda_r}f} {T_{\lambda_r}g}, \] we obtain the quotient identity. \begin{theorem}[General quotient rule] Assume \[ T_{\lambda_r}g\neq0, \qquad T_{\lambda_s}g\neq0. \] Then \[ D_U\!\left(\frac{f}{g}\right) = \frac{ T_{\lambda_r}g\,D_Uf - T_{\lambda_r}f\,D_Ug - T_{\lambda_r}g\,\mathcal L_{r,s}(f/g,g) }{ T_{\lambda_r}g\, T_{\lambda_s}g }. \] \end{theorem} The formula reduces to a simpler form whenever the corresponding Leibniz defect vanishes. \begin{corollary}[Defect-free quotient rule] If \[ \mathcal L_{r,s}(f/g,g)=0, \] then \[ D_U\!\left(\frac{f}{g}\right) = \frac{ T_{\lambda_r}g\,D_Uf - T_{\lambda_r}f\,D_Ug }{ T_{\lambda_r}g\, T_{\lambda_s}g }. \] \end{corollary} Thus the denominator is no longer \(g(x)^2\), as in ordinary calculus, but rather the product of two spectral evaluations, \[ g(\lambda_r x)\,g(\lambda_s x). \] The quotient rule therefore reflects the same two-point spectral structure already present in the product rule.

\section {Qualitative Classification of Second-Order Calculi}

\subsection{Horadam sequence}

Let $\W_n$ denote the Horadam sequence \cite{horadam1} defined by
\begin{equation}\label{eqn_horadam}
    \W_{n+2}=sW_{n+1}+t\W_n
\end{equation}
with initial conditions $\W_0=a$ y $\W_1=b$. Let
\[
\Delta=s^2+4t.
\]
When \(\Delta\neq 0\), the characteristic roots are denoted by
\[
\phi=\frac{s+\sqrt{s^2+4t}}2,\qquad
\varphi=\frac{s-\sqrt{s^2+4t}}2.
\]
The Binet formula for the Horadam sequence Eq.(\ref{eqn_horadam}) is
\begin{equation}\label{eqn_binet_h}
    \W_n=\frac{\alpha\varphi^{n}-\beta\phi^{n}}{\varphi-\phi}
\end{equation}
where $\alpha=b-a\phi$ and $\beta=b-a\varphi$. If $\Delta<0$, then we obtain the representation for the Horadam numbers
\begin{equation}\label{eqn_HDN}
    \W_n=r^n\left(b\frac{\sin(n\theta)}{\sin(\theta)}-a\frac{\sin((n-1)\theta)}{\sin(\theta)}\right)
\end{equation}
where $r=\Re(\phi)$ is the real part of $\phi$.

\subsection{Lucas calculus as Ward calculus of second order}

If we use the initial conditions $a=0$ and $b=1$ in the definition of the Horadam sequence, we obtain the Lucas sequence
\begin{equation}\label{lucas}
\{n+2\}_{s,t}=s\{n+1\}_{s,t}+t\{n\}_{s,t}.
\end{equation}
We define the Lucas-derivative $\D_{s,t}$ of the function $f(x)$ as
\begin{equation}
    \D_{s,t} f(x)=\frac{f(\varphi x)-f(\phi x)}{(\varphi-\phi)x}.
\end{equation}
Depending on the nature of the discriminant $\Delta$, we obtain the following classification.

\subsubsection{Hyperbolic Lucas calculus}

If $\Delta>0$, we obtain the hyperbolic Lucas calculus. This corresponds to the case where the characteristic polynomial has two distinct real roots
\[
\varphi\neq\phi,\qquad\varphi,\phi\in\R.
\]
The basic numbers in this calculus are
\begin{equation}
    \brk[c]{n}_{s,t}=\frac{\varphi^{n}-\phi^{n}}{\varphi-\phi}
\end{equation}
so all Lucas-type machinery \cite{pashaev,orozco2,orozco3} works quite naturally.
The quotient $q=\frac{\varphi}{\phi}$ typically fulfill $\vert q\vert<1$ so
\[
\brk[c]{n}_{s,t}=\varphi^{n-1}\frac{1-q^n}{1-q}=\varphi^{n-1}[n]_q
\]
thus, Lucas numbers are a deformation of $q$-numbers, and we obtain a $\varphi$-deformed $q$-calculus \cite{orozco}.

\subsubsection{Parabolic Lucas calculus}

This refers to the coalescence of the roots $\varphi$ and $\phi$,
\[
\varphi=\phi.
\]
Equivalently, $\Delta=0$. The basic numbers are no longer determined by a quotient of differences, but by the limit
\begin{equation}
    \brk[c]{n}_{s,t}=n\rho^{n-1}
\end{equation}
where $\rho=s/2$. This regime is extremely interesting because it contains several sub-geometries.

\medskip
\noindent
\textbf{Expansive regime.} ($\rho>1$) This regime corresponds to the study of pantograph-type functions. \cite{ise1,ise2}.

\medskip
\noindent
\textbf{Classic regime.} ($\rho=1$) Then $\brk[c]{n}_{s,t}=n$ and Newton’s calculus is reproduced exactly. This is the point where Lucas’s calculus and classical calculus intersect.

\medskip
\noindent
\textbf{Contractive regime.} ($\rho<1$) Here, the Lucas integers collapse
\[
\brk[c]{n}_{s,t}=n\rho^{n-1}\rightarrow0
\]

\subsubsection{Elliptic Lucas calculus}

This applies to the case where the roots form a complex conjugate pair
\[
\phi=\overline{\varphi}.
\]
Equivalently, $\Delta<0$. Wrinting
\[
\varphi=re^{i\theta},\qquad\phi=re^{-i\theta}
\]
the Lucas-numbers are
\begin{equation}
    \brk[c]{n}_{s,t}=r^{n-1}\frac{\sin(n\theta)}{\sin(\theta)}.
\end{equation}
The theory is no longer governed by pure growth, and oscillatory phenomena emerge.

\medskip
\noindent
\textbf{Regular cases.}
If
\[
\frac{\theta}{\pi}\notin\Q
\]
No additional zeros appear. The calculus remains regular.

\medskip
\noindent
\textbf{Singular cases.}
If
\[
\frac{\theta}{\pi}\in\Q
\]
There exist indices $m>0$ such that
\[
\brk[c]{m}_{s,t}=0.
\]
The derivative annihilates certain monomials. Singular phenomena, non-trivial kernels and possible nilpotent structures arise.

\subsection{Horadam calculus as non-Ward calculus of second order}

We now consider second-order recurrence calculi generated by Horadam sequences \[ \W_{n+2}=s\W_{n+1}+t\W_n, \] with \[ \W_0=a\neq 0. \] Since \[ \D_\W(1)=\W_0x^{-1}=ax^{-1}, \] every such calculus belongs to the non-Ward regime. Unlike the Ward case, the polynomial algebra is no longer invariant under the action of the associated recurrence derivative. Negative powers appear immediately and Laurent-type extensions become natural.

\subsubsection{Hyperbolic Horadam calculus} 

Assume \[ \Delta=s^2+4t>0. \] Then the characteristic polynomial possesses two distinct real roots \[ \phi\neq\varphi, \qquad \phi,\varphi\in\mathbb R. \] The associated recurrence derivative is \[ \D_\W f(x) = \frac{\alpha f(\phi x)-\beta f(\varphi x)} {(\phi-\varphi)x}. \] Since \[ \W_0=a\neq0, \] the lowering process passes through degree zero and continues toward negative powers. This class provides the simplest examples of regular non-Ward calculi.

\subsubsection{Elliptic Horadam calculus} 

Assume \[ \Delta<0. \] Then \[ \phi=re^{i\theta}, \qquad \varphi=re^{-i\theta}, \] and the recurrence admits the trigonometric representation Eq.(\ref{eqn_HDN}). The resulting calculus exhibits oscillatory behaviour and splits naturally into regular and singular subclasses.

\paragraph{Singular elliptic case.} 

Since \(r\neq0\) and \(\sin\theta\neq0\) in Eq.(\ref{eqn_HDN}), a zero of the sequence is equivalent to 
\begin{equation} 
b\sin(n\theta) = a\sin((n-1)\theta). 
\end{equation} 
Equivalently, \[ (b\cos\theta-a)\sin((n-1)\theta) + b\sin\theta\cos((n-1)\theta) = 0. \] 
Thus, unlike the Ward case, the location of the zeros depends not only on the spectral angle \(\theta\), but also on the initial conditions \(a\) and \(b\). 

Whenever \[ b\cos\theta-a\neq0, \] the zero condition may be written as 
\begin{equation} 
\tan((n-1)\theta) = -\frac{b\sin\theta} {b\cos\theta-a}.
\end{equation} 
Consequently, additional zero indices exist if and only if the arithmetic progression \[ \{(n-1)\theta:n\ge1\} \] intersects the set \[ \arctan\!\left( -\frac{b\sin\theta} {b\cos\theta-a} \right) +\pi\mathbb Z. \] 
\begin{definition} 
An elliptic non-Ward calculus is called singular when there exists an index \[ m>0 \] such that \[ b\sin(m\theta) = a\sin((m-1)\theta). \] Otherwise it is called regular. 
\end{definition} 
The singular elliptic non-Ward regime differs fundamentally from the Ward case. For Ward calculi (\(a=0\)), singularity depends only on the spectral angle through the condition \[ \sin(m\theta)=0. \] For non-Ward calculi, singularity depends simultaneously on the spectral parameter \(\theta\) and on the ratio of the initial values \[ \frac{a}{b}. \] Thus the zero structure is no longer determined solely by the characteristic roots but also by the choice of recurrence solution.

\paragraph{Regular elliptic case.} 

In contrast with the Ward case, the zero structure is no longer determined solely by the spectral angle. It depends simultaneously on the angle \(\theta\) and on the initial values \(a\) and \(b\). 
\begin{definition} 
An elliptic non-Ward calculus is called regular if \[ b\sin(n\theta) \neq a\sin((n-1)\theta) \] for every integer \(n\ge1\). \end{definition} 
In this situation, \(W_n\neq0 \qquad (n\ge1),\) and therefore no additional zero indices occur. Consequently, \(Z_W=\varnothing,\) and the associated recurrence derivative does not annihilate any positive-degree monomial.

\subsubsection{Parabolic Non-Ward Calculus (Formal Limit)} 

Although the present paper is restricted to recurrence sequences with distinct characteristic roots, it is instructive to examine the formal limit in which the two characteristic roots coalesce. Assume \(\Delta=0,\) and let \(\rho=\frac{s}{2}.\) The second-order recurrence \[ W_{n+2}=sW_{n+1}+tW_n \] then possesses a repeated characteristic root \(\rho\). The corresponding solutions are no longer described by the distinct-root Binet formula. Instead they take the confluent form \[ W_n=(A+Bn)\rho^n. \] 
Imposing the initial conditions \(W_0=a, W_1=b, \) gives \[ A=a, \qquad B=\frac{b}{\rho}-a, \] and therefore 
\begin{equation} 
W_n = \rho^n \left( a+n\left(\frac{b}{\rho}-a\right) \right). 
\end{equation} 
Since \(W_0=a\neq0,\) the corresponding calculus belongs to the non-Ward branch. The associated recurrence values vanish precisely when \[ a+n\left(\frac{b}{\rho}-a\right)=0. \] Hence a zero index exists if and only if 
\begin{equation}\label{eqn_SPnW} 
n= \frac{a\rho}{a\rho-b} 
\end{equation} 
is a positive integer. Consequently, the parabolic non-Ward regime admits both regular and singular behaviors: 
\begin{enumerate} 
\item Regular parabolic non-Ward calculus: Eq.(\ref{eqn_SPnW}) has no positive integer solution. 
\item Singular parabolic non-Ward calculus: Eq.(\ref{eqn_SPnW}) has a positive integer solution. 
\end{enumerate} 
Unlike the parabolic Ward case, where the basic numbers are \(n\rho^{n-1},\) the parabolic non-Ward calculus depends essentially on the additional parameter \(a=W_0\). In particular, \[ D_W(1)=a\,x^{-1}, \] so the lowering process continues through degree zero and naturally extends into a Laurent-type framework. The systematic study of these confluent non-Ward operators lies outside the scope of the present paper and belongs more naturally to a future theory of repeated characteristic roots.

\subsubsection{Complete classification of second-order calculi}

\begin{theorem}\label{theo_CCSC}
Let \[ W_{n+2}=sW_{n+1}+tW_n \] be a nonzero second-order recurrence sequence and let \(\mathcal D_W\) denote its associated recurrence derivative. Then the calculus generated by \(W\) is completely determined by the triple \[ \bigl(W_0,\Delta,Z_W\bigr), \] where \[ \Delta=s^2+4t \] is the discriminant of the characteristic polynomial and \[ Z_W=\{n\ge 0:W_n=0\} \] is the zero set of the recurrence. More precisely: 
\begin{enumerate} \item The value \(W_0\) determines the primary bifurcation \[ W_0=0 \Longleftrightarrow \text{Ward calculus}, \qquad W_0\neq0 \Longleftrightarrow \text{non-Ward calculus}. \] \item The discriminant \(\Delta\) determines the spectral geometry: \[ \Delta>0 \Longrightarrow \text{hyperbolic}, \qquad \Delta=0 \Longrightarrow \text{parabolic}, \qquad \Delta<0 \Longrightarrow \text{elliptic}. \] \item The zero set \(Z_W\) determines the regularity type: \[ Z_W=\varnothing \Longrightarrow \text{regular}, \qquad Z_W\neq\varnothing \Longrightarrow \text{singular}. \] In the Ward case the index \(0\) is excluded from this test. 
\end{enumerate} 
Consequently, the classification problem for second-order recurrence calculi reduces to the determination of the three invariants \[ (W_0,\Delta,Z_W). \]
\end{theorem} 
\begin{proof} 
The first assertion follows from the operator identity \[ \mathcal D_W(1)=W_0x^{-1}, \] which yields the Ward/non-Ward dichotomy. The second assertion follows from the classical classification of second-order recurrences according to the sign of the discriminant \(\Delta=s^2+4t\). The corresponding hyperbolic, parabolic and elliptic regimes were described in Sections~9.2 and~9.3. Finally, regularity and singularity are characterized by the presence or absence of additional zeros of the recurrence sequence. For Ward calculi this follows from the Lucas representations of Sections~9.2.1--9.2.3. For non-Ward calculi it follows from the Horadam zero criteria established in Sections~9.3.2 and~9.3.3. Since the three invariants \((W_0,\Delta,Z_W)\) are independent and each admits a unique value for a given recurrence, the resulting classification is exhaustive and mutually exclusive. \end{proof}

\begin{corollary}[Classification invariants] 
For second-order recurrence calculi, the triple \[ (W_0,\Delta,Z_W) \] completely determines the qualitative regime of a second-order calculus
\end{corollary} 
\begin{proof} Immediate from Theorem \ref{theo_CCSC}. 
\end{proof}

\begin{remark} The singular and parabolic subclasses appear in this section only because completeness of the second-order classification requires their identification. The present paper does not attempt to develop the corresponding operator theories. In particular, the analysis of additional zero indices, the geometry of the zero set, nilpotent structures, truncation phenomena, and the operator calculus associated with repeated characteristic roots lies beyond the scope of this work. Nevertheless, these exceptional subclasses play an important conceptual role: they reveal the boundary of validity of the distinct-root theory and provide a first indication of the singular and confluent calculi that arise naturally from the recurrence-operator framework. Their systematic study will be developed in future work. \end{remark}

\section{Future Directions} 

The results obtained in this paper suggest that recurrence-induced calculus should be viewed not merely as a collection of isolated examples, but as the beginning of a broader operator-theoretic classification program. The construction \[ U \longmapsto \mathcal B_U \longmapsto D_U \] associates an operator calculus to spectral recurrence data and reveals several independent structures that remain only partially understood. 

A first direction concerns the classification of recurrence operators. For second-order calculi, Theorem~\ref{theo_CCSC} shows that the qualitative regime is determined by the invariants \[ (W_0,\Delta,Z_W). \] Whether analogous finite systems of invariants exist for higher-order recurrences remains an open problem. 

A second direction concerns operator equivalence. Given two recurrence sequences \(U\) and \(V\), one may ask under which conditions the associated operators \(D_U\) and \(D_V\) should be regarded as equivalent. This leads naturally to the search for spectral, algebraic, and functional invariants capable of classifying recurrence operators up to conjugacy or isomorphism. 

A third direction concerns the structure of zero-defect pairs. The Leibniz defect defines families \[ \mathcal Z_{r,s}(U) = \{(f,g):E_{r,s}(f,g)=0\}, \] whose geometry remains largely unexplored. The distinction between spectral cancellation in the Ward regime and spectral balance in the non-Ward regime suggests the existence of nontrivial classes of admissible pairs beyond the examples considered here. 

The non-Ward theory itself leads to a separate research program. Because \[ D_U(1)=U_0x^{-1}, \] negative powers arise immediately whenever \(U_0\neq0\). This indicates that Laurent and bilateral frameworks provide a more natural setting than ordinary polynomial algebras for the non-Ward branch. The systematic development of this theory will be undertaken elsewhere. 

The classification developed in Section~9 also exposes two exceptional regimes lying at the boundary of the present theory. The first is the singular regime, characterized by the existence of additional zero indices. In this situation the zero set \[ Z_U=\{n:U_n=0\} \] becomes an invariant of the calculus and may govern truncation phenomena, kernel structure, and nilpotent behavior. The second is the confluent regime associated with repeated characteristic roots. The appearance of parabolic calculi in the second-order classification suggests a broader operator theory involving Jordan-type spectral structures and confluent Binet operators. 

Taken together, these observations suggest a hierarchy of recurrence calculi organized by operatorial, spectral, algebraic, and singular invariants. The present paper provides the distinct-root foundations of this program, while the bilateral, singular, and confluent theories remain for future investigation.

\end{document}